\documentclass[a4paper, 11pt, twoside]{article}

\pdfoutput=1

\usepackage{comment}
\usepackage{cite}
\usepackage[utf8]{inputenc}
\usepackage{enumerate}
\usepackage[ngerman, english]{babel}
\usepackage{array}
\usepackage{amsmath}
\usepackage{amssymb}
\usepackage{amsthm}
\usepackage{mathtools}
\usepackage{graphicx} 
\usepackage{setspace}
\usepackage{color}
\usepackage{scrlayer-scrpage}
\usepackage{titlesec}
\usepackage{tikz}
\usepackage[pdftex]{hyperref}
\usepackage[a4paper, lmargin=3.5cm, rmargin=3.0cm, top=3.0cm, bottom=2.5cm, head=14.5pt]{geometry}

 \hypersetup{
	pdfauthor = {Mareike Wolff},
	colorlinks = true,  
	linkcolor = black,
	citecolor = black,
	}		
  
\newcommand\dC{\mathbb{C}}

\newcommand\dN{\mathbb{N}}

\newcommand\dZ{\mathbb{Z}}

\newcommand\cF{\mathcal{F}}

\newcommand\cJ{\mathcal{J}}

\newcommand\cT{\mathcal{T}}

\newcommand\epsi{\varepsilon}

\DeclareMathOperator{\re}{Re}
\DeclareMathOperator{\im}{Im}

\DeclareMathOperator{\meas}{meas}
\DeclareMathOperator{\dist}{dist}

\newtheorem{thm}{Theorem}
\newtheorem{lemma}[thm]{Lemma}

\theoremstyle{definition}

\numberwithin{equation}{section}

\cohead{\small Julia sets of positive finite measure}
\cehead{\small Mareike Wolff}

\mathtoolsset{showonlyrefs=true}

\begin{document}
\title{Transcendental Julia sets of positive finite Lebesgue measure}
\author{Mareike Wolff}
\date{}
\maketitle

\begin{abstract}
\noindent We show that there exists a transcendental entire function whose Julia set has positive finite Lebesgue measure.\\


\end{abstract}

\section{Introduction}

For an entire function $f$, let $f^n$ denote the $n$th iterate of $f$. The \textit{Fatou set}, $\cF(f)$, consists of all $z$ such that the iterates $f^n$ form a normal family in a neighbourhood of $z$. Its complement, $\cJ(f)$, is called the \textit{Julia set}. Clearly, $\cF(f)$ is open and $\cJ(f)$ is closed. Moreover, $\cJ(f)$ is always non-empty, and either $\cJ(f)=\dC$ or $\cJ(f)$ contains no interior points. An introduction to iteration theory of transcendental entire functions can be found in \cite{bergweiler_survey}.

In 1987, McMullen \cite{mcmullen} showed that while the Julia set of $\lambda e^z$ for $\lambda\in(0,1/e)$ has Lebesgue measure zero, the Julia set of $\sin(az+b)$ always has positive Lebesgue measure. Around the same time, Eremenko and Ljubich \cite{eremenko-ljubich} constructed a transcendental entire function whose Julia set has positive Lebesgue measure using approximation theory. For both of these examples of Julia sets with positive Lebesgue measure, it can be seen that the measure of the Julia set is in fact infinite. Since then, various authors have studied the Lebesgue measure of Julia sets. Results on Julia sets of Lebesgue measure zero are given in \cite{eremenko-ljubich2, stallard, jankowski, zheng, w2}. Julia sets of positive Lebesgue measure are treated in \cite{bergweiler, aspenberg-bergweiler, bergweiler-chyzhykov, sixsmith}.  For some entire functions, the Julia set can be seen to be very large in the sense that it does not only have infinite measure but one can bound the size of its complement, the Fatou set. Schubert \cite{schubert} showed that the Lebesgue measure of the Fatou set of $\sin(z)$ is finite in any vertical strip of width $2\pi$. Hemke \cite{hemke} and the author \cite{w1} gave examples of transcendental entire functions whose Fatou set has finite measure. Now the natural question arises whether the Julia set of a transcendental entire function can also have positive measure and still be small in the sense that the Julia set itself has finite measure. In the paper at hand, we answer this question affirmatively.

\begin{thm}  \label{thm_main}  
There exists a transcendental entire function whose Julia set has positive finite Lebesgue measure.
\end{thm}

We will prove this in Section \ref{sec_proof} using approximation theory, following a similar approach as Eremenko and Ljubich in their paper \cite{eremenko-ljubich} mentioned above.

\section{Proof of Theorem \ref{thm_main}}  \label{sec_proof}

The proof uses the following result of Eremenko and Ljubich \cite[Main Lemma]{eremenko-ljubich}. Let $\dist(\cdot,\cdot)$ denote the Euclidean distance in $\dC$. 

\begin{lemma}  \label{lemma_el}
Let $(G_j)$ be a sequence of pairwise disjoint compact subsets of the complex plane such that $\dC\setminus G_j$ is connected for all $j\ge0$ and $\dist(G_j,0)\to\infty$ as $j\to\infty$. Let $(\epsi_j)$ be a sequence of positive real numbers and let $\Phi$ be holomorphic in a neighbourhood of $\bigcup_{j\ge0}G_j$. Then there exists an entire function $f$ such that 
\[\max_{z\in G_j}|f(z)-\Phi(z)|<\epsi_j\] 
for all $j\ge0$.
\end{lemma}

\begin{proof}[Proof of Theorem \ref{thm_main}]
The strategy of the proof is as follows. Using the approximation result stated in Lemma \ref{lemma_el}, we construct an entire function $f$ with the following properties:
\begin{itemize}
\item The function $f$ leaves a square $P_{0,0}$ invariant. By Montel's theorem, $P_{0,0}\subset\cF(f)$.
\item In addition, $f$ maps certain squares $P_{j,k}$ for $j,k\in\dZ$ into $P_{0,0}$. Then each $P_{j,k}$ is also contained in $\cF(f)$. The squares $P_{j,k}$ are chosen such that $\dC\setminus\bigcup_{j,k\in\dZ}P_{j,k}$ has finite Lebesgue measure.
\item There is a set of positive measure consisting of points which stay in the union of small squares $Q_{j}^3$ centred at the positive integers under iteration and can be shown to be in the Julia set.
\end{itemize}
See Figure \ref{fig_squaredistribution} for an illustration of the sets $P_{j,k}$ and $Q_j^3$. Let us now give the precise definition of these and some other squares.

\begin{figure}[ht]
\centering
\begin{tikzpicture}
\node at (0,0) {\tiny{x}};
\node [anchor=north] at (0,0) {\small{$0$}};

\foreach \j in {-1,...,3} \foreach \k in {-1,0,1} 
\draw [thick]({2*\j + 2^(-abs(\j)-abs(\k)-1)}, {2*\k + 2^(-abs(\j)-abs(\k)-1)}) rectangle ({2*\j+2-2^(-abs(\j)-abs(\k)-1)},{2*\k+2- 2^(-abs(\j)-abs(\k)-1)});
\node at (1,1) {$P_{0,0}$};
\node at (3,1) {$P_{1,0}$};
\node at (1,3) {$P_{0,1}$};

\foreach \j in {1,...,7} \fill ({\j-2^(-\j-2)}, {-2^(-\j-2)}) rectangle ({\j+2^(-\j-2)}, {2^(-\j-2)});
\node[anchor=east] at (0.95,0.1) {$Q_1^3$};
\node[anchor=east] at (2.05, 0.1) {$Q_2^3$};
\end{tikzpicture}

\caption{An illustration of the squares $P_{j,k}$ (white), which are contained in $\cF(f)$,  and $Q_{j}^3$ (black), which contain a subset of $\cJ(f)$ of positive measure.}  \label{fig_squaredistribution}
\end{figure}

For $j,k\in\dZ$, let $P_{j,k}$ be the square defined as
\[P_{j,k}:=\{z\in\dC:\,|\re z-(2j+1)|\le1-2^{-|j|-|k|-1},\,|\im z-(2k+1)|\le1-2^{-|j|-|k|-1}\}.\]
For $j\in\dN$, set 
\[\delta_j:=2^{-2j-6},\]
and define three nested squares as follows:
\begin{align}
Q_j^1&:=\{z\in\dC:\,|\re z-j|\le2^{-j-2}-2\delta_{j},\,|\im z|\le2^{-j-2}-2\delta_{j}\},\\
Q_j^2&:=\{z\in\dC:\,|\re z-j|\le2^{-j-2}-\delta_j,\,|\im z|\le2^{-j-2}-\delta_j\},\\
Q_j^3&:=\{z\in\dC:\,|\re z-j|\le2^{-j-2},\,|\im z|\le2^{-j-2}\}.
\end{align}
For each $j\in\dN$, divide $Q_j^3$ into $16$ equally sized squares $R_{j,1}^3,...,R_{j,16}^3$ of side length $2^{-j-3}$. Let $R_{j,l}^1$ and $R_{j,l}^2$ be the squares with the same centre as $R_{j,l}^3$ but side lengths $2^{-j-3}-4\delta_j$ and $2^{-j-3}-2\delta_j$, respectively. See Figure \ref{fig_squares_J} for an illustration of these squares. 

\begin{figure}
\centering
\begin{tikzpicture}
\draw[thick](-4,-4) rectangle (4,4);
\fill[black!35] (-3.8,-3.8) rectangle (3.8,3.8);
\fill[black!50] (-3.6,-3.6) rectangle (3.6,3.6);

\foreach \j in {-2,...,1}  \foreach \k in {-2,...,1} {
\draw[thick] ({2*\j},{2*\k}) rectangle ({2*\j+2},{2*\k+2});
\draw[thick, dashed] ({2*\j+0.2},{2*\k+0.2}) rectangle ({2*\j+1.8},{2*\k+1.8});
\draw[thick, dotted] ({2*\j+0.4},{2*\k+0.4}) rectangle ({2*\j+1.6},{2*\k+1.6});}

\node [anchor=east, align=left] at (-4.5,3.5) {large squares\\ (shaded)};
\draw[ultra thick,->] (-4.5,2.1)--(-3.85,2.1);
\node [anchor=east] at (-4.5,2.1) {$Q_j^3$};
\draw[ultra thick,->] (-4.5,0.1)--(-3.65,0.1);
\node [anchor=east] at (-4.5,0.1) {$Q_j^2$};
\draw[ultra thick,->] (-4.5,-1.9)--(-3.45,-1.9);
\node [anchor=east] at (-4.5,-1.9) {$Q_j^1$};

\node[anchor=west, align=left] at (4.5,3.5) {small squares\\ (bounded by lines)};
\draw[ultra thick,->] (4.5,1.9)--(3.85,1.9);
\node[anchor=west] at (4.5,1.9) {$R_{j,8}^3$};
\draw[ultra thick,->] (4.5,1.3)--(3.65,1.3);
\node[anchor=west] at (4.5,1.3) {$R_{j,8}^2$};
\draw[ultra thick, ->] (4.5, .7)--(3.45,.7);
\node[anchor=west] at (4.5,.7) {$R_{j,8}^1$};

\draw[ultra thick,->] (-3,4.5)--(-3,3.85);
\node[anchor=south] at (-3,4.5){$R_{j,1}^3$};
\draw[ultra thick, ->] (-1,4.5)--(-1,3.85);
\node[anchor=south] at (-1,4.5){$R_{j,2}^3$};
\end{tikzpicture}

\caption{An illustration of the squares $Q_j^m$ and $R_{j,l}^m$. The large square, $Q_j^3$, contains slightly smaller squares, $Q_j^2$ (light and dark grey) and $Q_j^1$ (dark grey). Moreover, $Q_j^3$ is divided into 16 squares $R_{j,l}^3$ by the solid lines, which contain smaller squares $R_{j,l}^2$ (dashed) and $R_{j,l}^1$ (dotted).}   \label{fig_squares_J}
\end{figure}
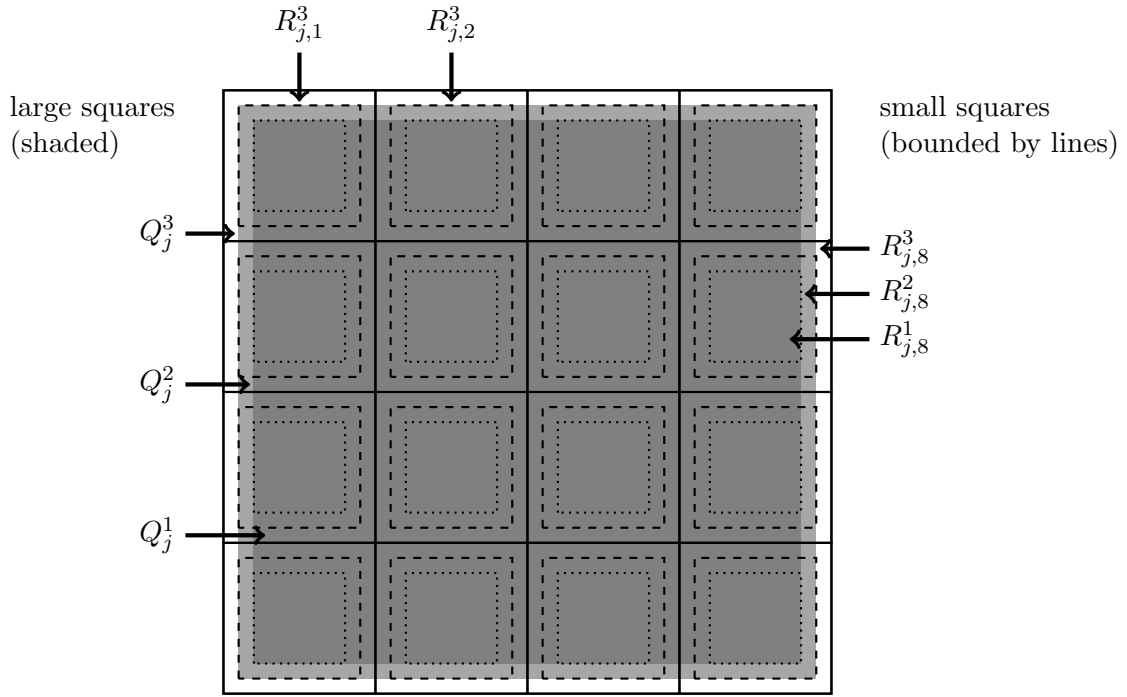

Let $\Phi_{j,l}$ be the affine map that maps $R_{j,l}^1$ onto $Q_{j+1}^2$ without rotation. By Lemma \ref{lemma_el}, there exists an entire function $f$ such that
\begin{enumerate}[(i)]
\item $|f(z)-(1+i)|<1/2$ for all $z\in P:=\bigcup_{j,k\in\dZ}P_{j,k}$;
\item $|f(z)-\Phi_{j,l}|<\delta_j^2$ for all $z\in R_{j,l}^2$ for each $j\in\dN$ and $l\in\{1,...,16\}$.
\end{enumerate}

By (i), $f$ is transcendental. 

Next, we show that $\cJ(f)$ has finite measure. By (i), $f(P)$ is contained in the interior of $P_{0,0}\subset P$. Montel's theorem yields that $P\subset\cF(f)$. For $j,k\in\dZ$, let $S_{j,k}$ be the square defined as 
\[S_{j,k}:=\{z\in\dC:\,|\re z-(2j+1)|\le1,\,|\im z-(2k+1)|\le1\}.\]
Then 
\begin{equation}
\begin{split}
\meas(S_{j,k}\setminus P_{j,k})&=4-4(1-2^{-|j|-|k|-1})^2=8\cdot2^{-|j|-|k|-1}-4\cdot4^{-|j|-|k|-1}\\
&<8\cdot2^{-|j|-|k|-1}=4\cdot2^{-|j|}\cdot2^{-|k|}.
\end{split}
\end{equation}
Since also $\bigcup_{j,k\in\dZ}S_{j,k}=\dC$, we obtain
\[\meas(\cJ(f))\le\meas(\dC\setminus P)=\sum_{j,k\in\dZ}\meas(S_{j,k}\setminus P_{j,k})\le4\sum_{j\in\dZ}2^{-|j|}\sum_{k\in\dZ}2^{-|k|}<\infty.\]

Following \cite{eremenko-ljubich}, we show that $\cJ(f)$ has positive measure. In order to do so, we construct a subset of $\cJ(f)$ as an intersection of nested sets. Recursively define collections $\cT_j$ of sets as follows. Set
\[\cT_1:=\{R_{1,1}^1\}.\]
Note that by (ii) and since $\delta_j^2<\delta_{j+1}$, we have $f(\partial R_{j,l}^1)\cap Q_{j+1}^1=\emptyset$ and $f(R_{j,l}^1)\cap Q_{j+1}^1 \ne \emptyset$ for all  $j\in\dN$ and $l\in\{1,...,16\}$. Thus $f(R_{j,l}^1)\supset Q_{j+1}^1\supset\bigcup_{m=1}^{16}R_{j+1,m}^1$. For $j\in\dN$, suppose that $\cT_j$ has been defined, and let 
\[\cT_{j+1}:=\{T_{j+1}:\,f^j(T_{j+1})\in\{R_{j+1,1}^1,....,R_{j+1,16}^1\},\,T_{j+1}\subset T_{j} \text{ for some }T_{j}\in\cT_{j}\}.\]
Set 
\[T:=\bigcap_{j\in\dN}\bigcup_{T_j\in\cT_j}T_j.\]
Now we show that $T\subset\cJ(f)$.

Let $z\in R_{j,l}^1$. Then by Cauchy's inequality and (ii),
\begin{equation}   \label{eq_f'approx}
|f'(z)-\Phi_{j,l}'(z)|\le\frac{1}{\delta_j}\sup_{z\in R_{j,l}^2}|f(z)-\Phi_{j,l}(z)|<\delta_j.
\end{equation}
Now $\Phi_{j,l}$ is the affine map that maps the square $R_{j,l}^1$, whose side length is $2^{-j-3}-4\delta_j$, onto $Q_{j+1}^2$, whose side length is $2\cdot (2^{-j-3}-\delta_{j+1})$. Thus,
\begin{equation}   \label{eq_Phi'}
\begin{split}
\Phi_{j,l}'(z)&=\frac{2\cdot(2^{-j-3}-\delta_{j+1})}{2^{-j-3}-4\delta_j}=\frac{2-2^{j+4}\delta_{j+1}}{1-2^{j+5}\delta_j}=\frac{2-2\cdot2^{j+5}\delta_j +2^{j+6}\delta_j-2^{j+4}\delta_{j+1}}{1-2^{j+5}\delta_j}\\
&=2+\frac{2^{j+4}(4\delta_j-\delta_{j+1})}{1-2^{j+5}\delta_j}
>2+\frac{2^{j+4}\cdot3\delta_j}{1-2^{j+5}\delta_j}>2+\delta_j.
\end{split}
\end{equation}
Hence, $|f'(z)|>2$. 

For $z\in T$, we have $f^j(z)\in\bigcup_{l=1}^{16}R_{j+1,l}^1$ for all $j\ge0$, and thus
\begin{equation}  \label{eq_fj'}
|(f^j)'(z)|=\prod_{m=0}^{j-1}|f'(f^m(z))|>2^j
\end{equation}
for all $j\in\dN$. Moreover, since $R_{j+1,l}^1\subset Q_{j+1}^1$, we have $|f^j(z)|=O(j)$ as $j\to\infty$.
For the spherical derivative $(f^j)^{\#}$ of $f^j$, this yields
\[(f^j)^\#(z)=\frac{2|(f^j)'(z)|}{1+|f^j(z)|^2}>\frac{2^{j+1}}{1+O(j^2)}\to\infty\]
as $j\to\infty$. By Marty's criterion, $T\subset\cJ(f)$.

Let us now show that $T$ has positive Lebesgue measure. For $z\in R_{j,l}^1$, by \eqref{eq_Phi'}, we have 
\[\re \Phi_{j,l}'(z)=\Phi_{j,l}'(z)>2+\delta_j,\] 
and thus $\re f'(z)>2>0$ by \eqref{eq_f'approx}. Therefore, $f$ is injective in $R_{j,l}^1$, and $f^{j}$ is injective in each $T_j\in\cT_j$. Using \eqref{eq_fj'}, $f^{j-1}(T_j)=R_{j,l}^1$ for some $l$, and that $f(R_{j,l}^1)\subset Q_{j+1}^3=\bigcup_{m=1}^{16}R_{j+1,m}^3$, we get
\begin{equation}
\begin{split}
\meas\left(T_j\setminus\bigcup_{T_{j+1}\in\cT_{j+1}}T_{j+1}\right)&\le\frac{1}{\min_{z\in T_j}|(f^{j})'(z)|^2}\meas\left(f^{j}\left(T_j\setminus\bigcup_{T_{j+1}\in\cT_{j+1}}T_{j+1}\right)\right)\\
&<2^{-2j}\meas\left(f(R_{j,l}^1)\setminus\bigcup_{m=1}^{16}R_{j+1,m}^1\right)\\
&\le2^{-2j}\meas\left(\bigcup_{m=1}^{16}R_{j+1,m}^3\setminus R_{j+1,m}^1\right)\\
&=2^{-2j}\sum_{m=1}^{16}(2^{-j-4})^2-(2^{-j-4}-4\delta_{j+1})^2\\
&=2^{-2j}\cdot16\cdot(2\cdot2^{-j-4}\cdot4\delta_{j+1}-16\delta_{j+1}^2)\\
&<2^{-2j}\cdot16\cdot2\cdot2^{-j-4}\cdot4\delta_{j+1}\\
&=2^{-3j+3}\delta_{j+1}=2^{-5j-5}.
\end{split}
\end{equation}
Moreover, $|\cT_j|=16^{j-1}$. Thus,
\begin{equation}
\begin{split}
\meas(R_{1,1}^1\setminus\cJ(f))&\le\meas(R_{1,1}^1\setminus T)\le\sum_{j=1}^\infty\sum_{T_j\in\cT_j}\meas\left(T_j\setminus\bigcup_{T_{j+1}\in\cT_{j+1}}T_{j+1}\right)\\
&\le\sum_{j=1}^\infty 16^{j-1}\cdot2^{-5j-5}=\sum_{j=1}^\infty2^{-j-9}=2^{-9}.
\end{split}
\end{equation}
On the other hand,
\[\meas(R_{1,1}^1)=(2^{-4}-4\delta_1)^2=(2^{-4}-4\cdot2^{-8})^2=2^{-8}(1-2^{-2})^2=\frac{9}{16}2^{-8}=\frac{9}{8}2^{-9}.\]
So $\meas(R_{1,1}^1)>\meas(R_{1,1}^1\setminus\cJ(f))$ and thus $\meas(\cJ(f))>0$.
\end{proof}

\paragraph{Acknowledgements} I would like to thank Weiwei Cui for drawing my attention to the topic by asking the question whether Julia sets of transcendental entire functions can have positive finite measure during the defence of my Ph.D. thesis. I also thank Walter Bergweiler for helpful remarks.

\bibliographystyle{plain}
\bibliography{Julia_positive_finite}

\end{document}